\documentclass[11pt]{article}
\usepackage{amssymb}
\usepackage{graphicx}
\usepackage{setspace}
  \usepackage{paralist}
  \usepackage{longtable}
   \usepackage{multirow}
    \usepackage{rotating}
\usepackage{marginnote}
\usepackage{mathrsfs}
\usepackage{amsmath, amsthm, amssymb}
\usepackage{authblk}
\usepackage{graphicx}
\usepackage{float}
\usepackage{hyperref}
\usepackage[margin=1in]{geometry}
\usepackage{comment}
\usepackage{color}
\usepackage{cite}
\usepackage{indentfirst}

\allowdisplaybreaks[4]
\numberwithin{equation}{section}
\date{}
\newtheorem{theorem}{Theorem}[section]
\newtheorem*{theorem*}{Theorem}
\newtheorem{proposition}[theorem]{Proposition}
\newtheorem{lemma}[theorem]{Lemma}
\newtheorem{corollary}[theorem]{Corollary}
\newtheorem{definition}[theorem]{Definition}

\theoremstyle{remark}
\newtheorem{remark}{Remark}
     
\newcommand{\dist}{\operatorname{dist}}
\newcommand{\diam}{\operatorname{diam}}
\newcommand{\intd}{\,\mathrm{d}}
\title{On the Dirichlet problem for fractional Laplace equation on a general domain}
\author{
Chenkai Liu
\thanks{the Institute of Mathematical Sciences, the Chinese University of Hong Kong; School of Mathematical Sciences and the Institute of Modern Analysis-A Frontier Research Center of Shanghai, Shanghai Jiao Tong University, China.}
Ran Zhuo \thanks{Department of Mathematical and Statistics, Huanghuai University, China. }
}

\begin{document}
\maketitle
\begin{abstract}
\noindent In this paper, we study the weak strong uniqueness of the Dirichlet type problems of fractional Laplace (Poisson) equations.
We construct the Green's function and the Poisson kernel.
We then provide a somewhat sharp condition for the solution to be unique.  We also show that
the solution under such condition exists and must be  given by our Green's function and Poisson kernel.
In doing these, we establish several basic and useful properties of the Green's function and Poisson kernel.
Based on these, we obtain some further a priori estimates of the solutions. Surprisingly those estimates are quite different from the ones for the local type elliptic equations such as Laplace equations.
These are basic properties to the fractional Laplace equations and can be useful in the study of related problems.


\noindent{\bf {MSC}}:
                      35A01; 
                      35B45; 
                      35J08; 
                      35S05.  
\end{abstract}

\section{Introduction and results}

We are interested in the existence and uniqueness of solutions for the following Dirichlet problem with fractional Laplacian:
\begin{equation}\label{pb+c}
\left\{
\begin{aligned}
(-\Delta)^s u+c u&=f\quad &\text{in}&\quad \Omega,\\
u&=0\quad&\text{on}&\quad\mathbb{R}^n\backslash \Omega,
\end{aligned}
\right.
\end{equation}
where $0<s<1$, $\Omega$ is a bounded domain in $\mathbb{R}^n$ satisfying the uniform exterior ball condition.

In the classical work of Laplacian, the Dirichlet problem of $W^{2,p}$-type plays an essential part: Assume $\Omega\subset \mathbb{R}^n$ be a bounded domain with smooth boundary,  $f\in L^p(B_1)$, $\vec b, c\in L^\infty(\Omega)$ and $c\geqslant0$ in $\Omega$, consider the problem \cite{MR1814364}:
\begin{equation}\label{pb0}
\left\{
\begin{aligned}
-\Delta u+\vec b\cdot\nabla u+cu&=f\ \ \ &\text{in}\ &\Omega,\\
u&=0\ \ \ &\text{on}\ &\partial \Omega.
\end{aligned}
\right.
\end{equation}
It is well known that \eqref{pb0} has a unique solution $u\in W^{2,p}(\Omega)\cap W^{1,p}_0(\Omega)$, which admits the following apriori estimte:
\begin{equation}
\|u\|_{W^{2,p}(\Omega)}\leqslant C\|f\|_{L^p(\Omega)}.
\end{equation}
Dirichlet problem has been utilized in nonlinear diffusion generated by nonlinear sources, thermal ignition of gases, gravitational equilibrium of stars and elsewhere, more details on the application of Dirichlet problem can be found in \cite{MR1625845,MR0151710,MR1814364,MR340701}. Many fruitful results on such problems for Laplacian were obtained. Brezis et al. \cite{MR509489} and Figueiredo et al. \cite{MR664341} (see also Berestycki and Lions \cite{MR600785} for a particular case) worked on a priori bound and existence for solutions of semilinear equations. Chen and Li \cite{MR1217588} investigated a priori estimates for solutions to nonlinear elliptic equations.

The fractional Laplacian has caught the researchers' attentions because of the applications in physics, astrophysics, mechanics and economics in recent years (see \cite{MR2105239,MR2465826,MR2806700,MR1654824,MR2451371} ). In contrast to the situation with local Laplacian, essential tools such as maximum principle, Harnack principle and Hopf lemma are not at our disposal when dealing with solutions in the setting of the nonlocal fractional Laplacian. Rather, these tools need to be reconstructed based on the nonlocal properties of fractional operators. For $c=0$ in \eqref{pb+c}, Ros-Oton and Serra \cite{MR3168912} studied the regularity up to the boundary of solutions to \eqref{pb+c} by developing fractional analog of the Krylov boundary Harnack method. As $\Omega$ being a ball,
$c=0$ in \eqref{pb+c}, Bucur \cite{MR3461641} provided the representation formulas of solutions for \eqref{pb+c}. For more results on the fractional Laplacian, please see \cite{MR3830727,MR4104934,MR2354493,MR2494809,MR4379307,MR3600062,MR2964681,MR4232665,MR4310545,MR3947481,MR3959488} and the references therein.

It is a common idea to represent the solutions of \eqref{pb+c} can as the
convolution of the Green function with the forcing term for $c=0$. The construction of Green functions on $\Omega$ plays an important role in our study.

It is known that (see for \cite{MR1892228})
\begin{equation}
\label{condclass}
  \lim_{\epsilon\rightarrow 0}\epsilon^{-1}\int_{\Omega\backslash\Omega_\epsilon}|u(x)|\intd x=0
\end{equation}
indicates that the trace of  $u$ is $0$ on $\partial\Omega$. Hence, \eqref{condclass} can be a uniqueness condition for the classical second order elliptic PDE.
Here we denote
\begin{equation}
  \Omega_\epsilon=\{x\in\Omega\mid\dist(x,\partial\Omega)>\epsilon\} =\{x\in\Omega\mid \overline{B_\epsilon(x)}\subset\Omega\},
\end{equation}

Similar to the Laplace problem, our main observation is that the following condition \eqref{cond0} plays an important role in the existence, uniqueness and the maximum principle for problem \eqref{pb+c}.
\begin{equation}
\label{cond0}
  \lim_{\epsilon\rightarrow 0}\epsilon^{-s}\int_{\Omega\backslash\Omega_\epsilon}|u(x)|\intd x=0.
\end{equation}
Indeed, we have
\begin{theorem}[Maximum Principle]
\label{thm:main-mp}
 Let $\Omega$ be a bounded domain in $\mathbb{R}^n$ satisfying the uniform exterior ball condition, $0<s<1$, $c\geqslant 0$ in $\Omega$. If $u\in\mathcal{L}_{2s}$ satisfies
 \begin{equation}
   \left\{
\begin{aligned}
(-\Delta)^s u+c u&\geqslant 0\quad &\text{in}&\quad \Omega,\\
u&\geqslant0\quad&\text{on}&\quad\mathbb{R}^n\backslash \Omega.
\end{aligned}
\right.
 \end{equation}
 \begin{equation}
   \label{cond1}
  \lim_{\epsilon\rightarrow 0}\epsilon^{-s}\int_{\Omega\backslash\Omega_\epsilon}u^-(x)\intd x=0.
 \end{equation}
 Then we conclude
 \begin{equation}
   u\geqslant 0\qquad\text{in}\quad \mathbb{R}^n.
 \end{equation}
\end{theorem}
\noindent Here, we say $\Omega$ satisfies the exterior ball condition if for any $X\in\partial\Omega$, there exists a ball $B_r(c)$ such that
\begin{equation}
  B_r(c)\cap\Omega=\emptyset\qquad\text{and}\qquad X\in\overline{B_r(c)}\cap\overline{\Omega}.
\end{equation}
We say $\Omega$ satisfies the uniform exterior ball condition with uniform radius $r$, if the radius of exterior ball corresponding to any $X\in\partial\Omega$ can be chosen uniformly as $r$.

\begin{theorem}
\label{thm:main1}
  Let $\Omega$ be a bounded domain in $\mathbb{R}^n$ satisfying the uniform exterior ball condition, $0<s<1$, $c=0$ in $\Omega$. If $f\in L^1_s(\Omega)$, then problem \eqref{pb+c} has a unique solution $u=G_\Omega\ast f$ satisfying condition \eqref{cond0}.
\end{theorem}
\noindent Here and hereafter, the weighted Lebesgue spaces $L^p_\sigma$ is defined as:
\begin{equation}
L^p_\sigma(\Omega)=\{f\in L^1_{\mathrm{loc}}(\Omega)\mid[\dist(\cdot, \mathbb{R}^n\backslash\Omega)^\sigma f]\in L^p(\Omega)\},
\end{equation}
with a nature norm:
$\|f\|_{L^p_\sigma(\Omega)}=\|\dist(\cdot, \mathbb{R}^n\backslash\Omega)^\sigma f\|_{L^p(\Omega)}$.

\begin{remark}
  If condition \eqref{cond0} is not satisfied, then the solution is not unique. Indeed, even for $\Omega =B_1$, $f\equiv 0$ and $g\equiv 0$, there exists a nontrivial solution
  \begin{equation}
    u(x)=\begin{cases}
      (1-|x|^2)^{s-1}\qquad&\text{if}\quad x\in B_1,\\
      0\qquad&\text{if} \quad x\in \mathbb{R}^n\backslash B_1.
    \end{cases}
  \end{equation}
  One observe that such $u$ nearly satisfies the condition \eqref{cond0}.
\end{remark}

This kind of weak-strong uniqueness and very weak non-uniqueness phenomena is also found for other equations such as the Navier-Stokes equations and some transport equations. For the Navier-Stokes equations, the weak-strong uniqueness was  established by Fabes, Ladyzhenskaya, Prodi, Serrin and etc. (see for \cite{MR316915},\cite{MR0236541},\cite{MR126088},\cite{MR0150444},$\cdots$). Non-uniqueness of weaker solutions was first proved by Buckmaster in \cite{MR3898708}, see also \cite{MR3951691,MR4422213}.
Recently, Cheskidov and Luo showed  the critical regularity criteria for the non-uniqueness of very weak solution to Navier-Stokes equations \cite{MR4462623}\cite{2105.12117} and  the transport equation in \cite{MR4199851}.

\begin{theorem}
\label{thm:main+c}
  Let $\Omega$ be a bounded domain in $\mathbb{R}^n$ satisfying the uniform exterior ball condition, $0<s<1$, $|\sigma|\leqslant s$, $c\geqslant0$ in $\Omega$. If $f\in L^p_\sigma(\Omega)$, then problem \eqref{pb+c} has  a unique solution $u$ satisfying condition \eqref{cond0}.

Moreover, the solution $u$ satisfies the following estimate:
\begin{enumerate}
\item If $p=1$, then for any $q$ satisfying $1\leqslant q<\frac{n}{n-2s}$,
$u\in  L^q_\sigma(\Omega)$ with
\begin{equation}
\|u\|_{ L^q_\sigma(\Omega)}\leqslant C\|f\|_{ L^1_\sigma(\Omega)}.
\end{equation}
\item If $1<p<\frac{n}{2s}$, then $u\in  L^{\frac{np}{n-2sp}}_\sigma(\Omega)$ with
\begin{equation}
\|u\|_{ L^{\frac{np}{n-2sp}}_\sigma(\Omega)}\leqslant C\|f\|_{ L^p_\sigma(\Omega)}.
\end{equation}
\item If $p>\frac{n}{2s}$, then $u\in  L^\infty_\sigma(\Omega)$ with
\begin{equation}
\|u\|_{ L^\infty_\sigma(\Omega)}\leqslant C\|f\|_{ L^p_\sigma(\Omega)}.
\end{equation}
\end{enumerate}
\end{theorem}
The following well known result can also be seen as a corollary of our estimate in Theorem \ref{thm:main+c}:
\begin{corollary}
Let $\Omega$ be a bounded domain in $\mathbb{R}^n$ satisfying the uniform exterior ball condition, $0<s<1$, $c\geqslant0$ in $\Omega$.
  If $f\in L^\infty(\Omega)$ in \eqref{pb+c}, and $u$ satisfies \eqref{cond0} then $u\in C^s(\Omega)$.
\end{corollary}

We organize this paper as follows. In Section \ref{sec.2} we introduce some basic preliminaries about the fractional Laplacian.  Section \ref{sec.3} is devoted to the construction of Green's function on general domains. Some basic properties and estimates are established in Section\ref{sec.4}. We construct the Poisson kernel in Section \ref{sec.5}. And we finally solve the Dirichlet problem in Section \ref{sec.6}.

The authors extend their thanks to Professor Congming Li under whose guidance they carried out this work.

\section{Preliminaries about the fractional Laplacian }
\label{sec.2}
For $0<s<1$ and $u\in C_0^\infty(\mathbb{R}^n)$, the fractional Laplacian $(-\Delta)^su$ is given by
\begin{equation}\label{1.0}
(-\Delta)^su(x)=C_{n,s}{\rm P.V.}\int_{\mathbb{R}^n}\frac{u(x)-u(y)}{|x-y|^{n+2s}}\intd y,\ \ \ 0<s<1,
\end{equation}
where P.V. stands for the Cauchy principle value (see for \cite{MR0350027}).
In order that $(-\Delta)^sw$ make sense as a distribution, it is common to define:
\begin{equation*}
\mathcal{L}_\alpha=\bigg\{u: \mathbb{R}^n\rightarrow\mathbb{R} \bigg|\int_{\mathbb{R}^n}\frac{|u(y)|}{1+|y|^{n+\alpha}}\intd y<+\infty\bigg\}.
\end{equation*}
It is easy to see that for $w\in \mathcal{L}_{2s}$, $(-\Delta)^sw$  as a distribution is well-defined: $\forall \varphi\in C_0^\infty(\mathbb{R}^n)$,
\begin{equation}
(-\Delta)^sw(\varphi)=\int_{\mathbb{R}^n}w(x)(-\Delta)^s\varphi(x) \intd x.
\end{equation}

The fundamental solution $\Phi$ for $s$-Laplacian is defined as:
\begin{equation}
  \Phi(x)=a(n,s)|x|^{2s-n}.
\end{equation}
 One has the following equation in sense of distribution:
 \begin{equation}
   (-\Delta)^s \Phi=\delta_0.
 \end{equation}

The Green's functions and Poisson kernels of fractional Laplacian on the unit ball is introduced in \cite{MR3461641}.
Define the Poisson kernel $P(\cdot,\cdot)$ and Green's function $G(\cdot,\cdot)$ for $s$-Laplacian in $B_1$ \cite{MR3461641}  as:
\begin{equation}\label{P}
P(x,y)=c(n,s)\left(\frac{1-|x|^2}{|y|^2-1} \right)^s\frac{1}{|x-y|^n}, \ \ \text{for}\ x\in B_1, y\in\mathbb{R}^n\backslash B_1.
\end{equation}
\begin{equation}\label{G}
G(x,y)=\frac{\kappa(n,s)}{|x-y|^{n-2s}}\int_0^{\rho(x,y)} \frac{t^{s-1}}{(1+t)^\frac{n}{2}}\ \intd t, \ \ \text{for}\ x,y\in B_1,
\end{equation}
where $\rho(x,y)$ is defined as
\begin{equation}
\rho(x,y)=\frac{(1-|x|^2)(1-|y|^2)}{|x-y|^2}.
\end{equation}

Let $P\ast g$ and $G\ast f$ be defined as
\begin{equation}
  \label{Pg}
 P\ast g(x)=\begin{cases}
\displaystyle \int_{\mathbb{R}^n\backslash B_1} P(x,y)g(y)\intd y&\quad x\in  B_1,\\
g(x)&\quad x\in\mathbb{R}^n\backslash B_1,
\end{cases}
\end{equation}

\begin{equation}\label{Gf}
G\ast f(x)=\begin{cases}
\displaystyle \int_{B_1} f(y)G(x,y)\intd y&\quad x\in  B_1,\\
0&\quad x\in\mathbb{R}^n\backslash B_1.
\end{cases}
\end{equation}

The following two results are introduced in \cite{MR3461641}:
\begin{lemma}\label{exi0}
Let $0<s<1$, $f\in C^{2s+\epsilon}(B_1)\cap C(\overline{B_1})$. Then there exists a unique continuous solution: $u=G\ast f$ that solves the following problem \eqref{pb1} pointwisely.
\begin{equation}\label{pb1}
\left\{
\begin{aligned}
(-\Delta)^s u&=f\quad &\text{in}&\quad B_1,\\
u&=0\quad&\text{on}&\quad\mathbb{R}^n\backslash B_1.
\end{aligned}
\right.
\end{equation}
\end{lemma}
\begin{lemma}\label{exi1}
Let $0<s<1$, $g\in \mathcal{L}_{2s}\cap C(\mathbb{R}^n)$. Then there exists a unique continuous solution: $u=P\ast g$ that solves the following problem \eqref{pb2} pointwisely.
\begin{equation}\label{pb2}
\left\{
\begin{aligned}
(-\Delta)^s u&=0\quad &\text{in}\ &B_1,\\
u&=g\quad&\text{on}\ &\mathbb{R}^n\backslash B_1.
\end{aligned}
\right.
\end{equation}
\end{lemma}

The classical maximum principle for fractional Laplacian is given in \cite{MR2270163} Lemma 2.17:
\begin{theorem}[Maximum principle]
\label{thm:silvestre-mp}
  Let $\Omega$ be a bounded domain in $\mathbb{R}^n$, $u$ be a lower semi-continuous function in $\overline{\Omega}$ such that
  \begin{equation}
  \begin{cases}
    (-\Delta)^s u\geqslant 0&\qquad\text{in}\quad\mathcal{D}^\prime(\Omega),\\
    u\geqslant 0\qquad\text{in}\quad\mathbb{R}^n\backslash\Omega.
    \end{cases}
  \end{equation}
  Then $u\geqslant 0$ in $\mathbb{R}^n$.
\end{theorem}
The following technique also comes from \cite{MR2270163}:
 \begin{definition}
   Let $\Gamma$ be defined as follow
   \begin{equation}
     \Gamma(x)=\left\{
     \begin{aligned}
       &\Phi(x)=a(n,s)|x|^{2s-n}\quad&|x|\geqslant 1,\\
       &\frac{n-2s+2}{2}a(n,s)-\frac{n-2s}{2}a(n,s)|x|^2\quad&|x|<1.
     \end{aligned}\right.
   \end{equation}
 \end{definition}
 It is easy to check: $\Gamma\in C^{1,1}(\mathbb{R}^n)\cap\mathcal{L}_{2s}$. Therefore, one can define $\gamma=(-\Delta)^s\Gamma\in C(\mathbb{R}^n)$.
\begin{proposition}[Proposition 2.11 in \cite{MR2270163}]
\label{prop:gamma}
$\gamma\geqslant0$ and
\begin{equation}\int_{\mathbb{R}^n}\gamma(x) \intd x=1.\end{equation}
\end{proposition}
\begin{proposition}[Proposition 2.13 in \cite{MR2270163}]
  Define $\gamma_\lambda(x)=\frac{1}{\lambda^n}\gamma(\frac{x}{\lambda})$. Then $\gamma_\lambda$ is an approximation of Dirac $\delta$ as $\lambda\rightarrow 0^+$.
\end{proposition}

\begin{theorem}[Proposition 2.15 in \cite{MR2270163}]
\label{thm:silvestra-semi}
 Let $\Omega$ be a domain in $\mathbb{R}^n$. If $u\in\mathcal{L}_{2s}$ satisfies
  \begin{equation}
    (-\Delta)^s u\leqslant0\quad\text{in}\quad \mathcal{D}^\prime(\Omega).
  \end{equation}
  Then, there exists a verson $\tilde{u}$ of $u$ (i.e. $u=\tilde{u}$ a.e. $\mathbb{R}^n$), such that
   $\tilde{u}$ is upper semi-continuous in $\Omega$.

   Indeed,
   \begin{equation}
     \tilde{u}=\lim_{\lambda\rightarrow 0^+}u\ast\gamma_{\lambda}.
   \end{equation}
\end{theorem}

The following result proved in \cite{MR4129482} is also crucial in our progress.
\begin{theorem}[Theorem 5.4 in \cite{MR4129482}]
\label{thm:maxcons}
Let $\Omega$ be a domain in $\mathbb{R}^n$. Assume $u,v\in\mathcal{L}_{2s}$, $f, g\in L^1_{\rm loc}(\Omega)$, and satisfy
\begin{equation}
\left.\begin{aligned}
& (-\Delta)^s u(x)+\vec{b}(x)\cdot\nabla u(x)+c(x)u(x)\leqslant f(x)\\
& (-\Delta)^s v(x)+\vec{b}(x)\cdot\nabla v(x)+c(x)v(x)\leqslant g(x)
\end{aligned}\right\} \qquad \text{in}\quad \mathcal{D}'(\Omega),
\end{equation}
where $\|\vec b(x)\|_{C^1(\Omega)}+\|c(x)\|_{L^\infty(\Omega)}<\infty$.  Then for $w(x)=\max\{u(x),v(x)\}$, it holds that
\begin{equation}
(-\Delta)^sw(x)+\vec{b}(x)\cdot\nabla w(x)+c(x)w(x)\leq f(x)\chi_{u> v} + g(x)\chi_{u<v}+\max\{f(x),g(x)\}\chi_{u=v}\ {\rm in}\ \mathcal{D}'(\Omega).
\end{equation}
\end{theorem}

\section{The construction of Green function for general domains}
\label{sec.3}
In this  section, we will use Perron's method to show the existence of Green functions for domains with sufficiently regular boundary. Indeed, we aim to prove the following theorem:
\begin{theorem}
\label{thm:ex-1}
  Let $\Omega$ be a bounded domain in $\mathbb{R}^n$ satisfying the uniform exterior ball condition. Assume that $0<s<1$. Then for any $y\in\Omega$, there exists a function $h_y\in C(\mathbb{R}^n)\cap C^{\infty}(\Omega)$ such that
  \begin{equation}
  \label{pbg}
  \left\{
    \begin{array}{llr}
      (-\Delta)^sh_y=0&\quad\text{in}\quad&\mathcal{D}^\prime(\Omega),\\
      \displaystyle h_y=g_y&\quad\text{on}\quad& \mathbb{R}^n\backslash\Omega,
    \end{array}\right.
  \end{equation}
  for $g_y(x)=\frac{a(n,s)}{|x-y|^{n-2s}}$.
\end{theorem}
Mollifier is a basic tool in our analysis. In the following, we denote by $J_\epsilon [u]$ the mollification of $u$:
\begin{equation}
J_\epsilon [u](x)=(j_\epsilon\ast u)(x)=\int_{B_{\epsilon}(x)} j_\epsilon(x-y)u(y)\intd y,
\end{equation}
where $j\in C_0^{\infty}(B_1)$ is a positive smooth radially symmetric function supported in $B_1$ satisfying $\int_{\mathbb{R}^n}j(x)\ dx=1$. Moreover, $j_\epsilon(x)=\frac{1}{\epsilon^n}j(\frac{x}{\epsilon})$.

\begin{definition}
\label{def:GO}
With $h_y$ defined, one can define the Green function corresponding to the domain $\Omega$ as:
\begin{equation}
  G_\Omega(x,y)=\Phi(x-y)-h_y(x).
\end{equation}
\end{definition}
\begin{remark}
 $G_\Omega(x,y)=0$, if either one of $x,y$ is not in $\Omega$.
\end{remark}

In order to prove Theorem~\ref{thm:ex-1}, we need the following definitions and lemmas to apply the Perron's method.

\begin{definition}
We define the Perron sub-solution class corresponding to \eqref{pbg} as:
  \begin{equation}
    \mathcal{P}=\left\{p\in C(\mathbb{R}^n)\mid
      \left\{\begin{array}{lr}
      (-\Delta)^sp\leqslant0\quad&\text{in}\ \mathcal{D}^\prime(\Omega),\\
      \displaystyle p\leqslant g_y\quad&\text{on}\ \mathbb{R}^n\backslash\Omega,\\
      p\geqslant0\quad&\text{in}\ \mathbb{R}^n.\\
      \end{array}\right. \right\}.
  \end{equation}
\end{definition}
\begin{remark}
Set $\mathcal{P}$ is nonempty.

  Indeed,  let $p_0:\mathbb{R}^n\rightarrow\mathbb{R}$ be defined as $p_0(x)\equiv 0$. Then, one easily verifies $p_0\in\mathcal{P}$.
\end{remark}
\begin{remark}\label{rmk2}
By maximum principle~\ref{thm:silvestre-mp}, each $p\in\mathcal{P}$ satisfies, for all $x\in\mathbb{R}^n$
\begin{equation}
0\leqslant p(x)\leqslant g_y(x)\quad
  \text{and}\quad 0\leqslant p(x)\leqslant \max_{\xi\in\mathbb{R}^n\backslash\Omega}g_y(\xi).
\end{equation}
\end{remark}
\begin{remark}
\label{rmk3}
  The Perron sub-solution class $\mathcal{P}$ is separable, with respect to the metric\begin{equation}
    \rho(p_1,p_2)=\max_{x\in\mathbb{R}^n}|p_1(x)-p_2(x)|=\|p_1-p_2\|_{C(\mathbb{R}^n)}.
  \end{equation}
  Indeed, $\mathcal{P}\subset C_0({\mathbb{R}^n})$ (the Banach space consists of continuous functions on $\mathbb{R}^n$ which vanishes at infinity with norm $\|\cdot\|_{C(\mathbb{R}^n)}$) and $C_0({\mathbb{R}^n})$ is separable.
\end{remark}
\begin{definition}
The Perron solution $S$ of \eqref{pbg} is defined to be:
  \begin{equation}
S(x)=\sup_{p\in\mathcal{P}}p(x).
  \end{equation}
\end{definition}
\begin{remark}
Seeing Remark~\ref{rmk2}, we have the following estimate on $S$:
\begin{equation}
0\leqslant S(x)\leqslant g_y(x)\quad
  \text{and}\quad 0\leqslant S(x)\leqslant \max_{\xi\in\mathbb{R}^n\backslash\Omega}g_y(\xi)\quad\text{for\ all}\ x\in\mathbb{R}^n.
\end{equation}
\end{remark}
\begin{lemma}
\label{lem:ex-1}
$S$ is lower semi-continuous.
\end{lemma}
\begin{proof}[Proof of Lemma \ref{lem:ex-1}]
Indeed, for any $x\in\mathbb{R}^n$, $\epsilon>0$, there exists $p\in\mathcal{P}$ such that
    \begin{equation}
    p(x)>S(x)-\epsilon.
    \end{equation}
    However, $p$ is continuous, then there exists $\delta>0$ such that for any $\xi\in B_\delta(x)$,
    \begin{equation}
      p(\xi)>p(x)-\epsilon.
    \end{equation}
    Therefore, for any $\xi\in B_\delta(x)$
    \begin{equation}
      S(\xi)\geqslant p(\xi)>p(x)-\epsilon>S(x)-2\epsilon.
    \end{equation}
    I.e. $S$ is lower semi-continuous.
\end{proof}
\begin{lemma}
\label{lem:ex-4}
  Let $p_1,p_2\in\mathcal{P}$, $p_3$ defined as follow:
  \begin{equation}
    p_3(x)=\max\{p_1(x),p_2(x)\}\quad\text{for\ any}\quad x\in\mathbb{R}^n.
  \end{equation} Then $p_3\in\mathcal{P}$.
\end{lemma}
\begin{proof}
 It is easy to see: $p_3\in C(\mathbb{R}^n)\cap \mathcal{L}_{2s}$ from the fact that $p_1,p_2\in C(\mathbb{R}^n)\cap \mathcal{L}_{2s}$.

 It is also clear that $p_3\leqslant g_y$ on $\mathbb{R}^n\backslash\Omega$.

 Choose $\vec b, c, f, g$ to be zero in Theorem \ref{thm:maxcons}, one derives
\begin{equation}
  (-\Delta)^s p_3\leqslant 0\quad\text{in}\quad\mathcal{D}^\prime(\Omega).
\end{equation}

Therefore, $p_3\in\mathcal{P}$.
\end{proof}

\begin{lemma}
   Let $\Omega$ be a bounded domain in $\mathbb{R}^n$ satisfying the exterior ball condition, $S$ be the Perron solution corresponding to \eqref{pbg}. Then, there exists a monotonically increasing sequence $\{p_i\}_{i=1}^\infty\subset \mathcal{P}$  such that
  \begin{equation}
   S(x)=\lim_{i\rightarrow\infty}p_i(x).
  \end{equation}
\end{lemma}
\begin{proof}
 From Remark \ref{rmk3}, there exists a countable subset $\{p^*_j\}_{j=1}^\infty$ of $\mathcal{P}$ that is dense in $\mathcal{P}$. As a consequence, we can write $S(x)$ as:
 \begin{equation}
   S(x)=\sup_{j\in\mathbb{N}_+} p^*_j(x).
 \end{equation}

Define
\begin{equation}
  p_i(x)=\max_{1\leqslant j\leqslant i}p^*_j(x).
\end{equation}
Clearly, sequence $\{p_i\}_{i=1}^\infty$ is monotonically increasing, and
 \begin{equation}
   S(x)=\lim_{i\rightarrow\infty}p_i(x).
 \end{equation}
 Moreover, by Lemma \ref{lem:ex-4}, one derives $p_i\in\mathcal{P}$.
\end{proof}
\begin{corollary}
\label{cor-ext1}
  $S(x)$ satisfies
  \begin{equation}
    (-\Delta)^s S\leqslant0\quad\text{in}\quad \mathcal{D}^\prime(\Omega).
  \end{equation}
\end{corollary}
\begin{proof}
For any test function $\varphi\in\mathcal{D}(\Omega)$, there exist constants $C_1,C_2,M>0$ such that
\begin{equation}
\begin{aligned}
  |(-\Delta)^s \varphi(x)|\leqslant C_1&\quad\text{for} \quad x\in\mathbb{R}^n\\
  |(-\Delta)^s \varphi(x)|\leqslant \frac{C_2}{|x|^{n+2s}}&\quad\text{for} \quad |x|\geqslant M.
\end{aligned}
\end{equation}
Then
\begin{equation}
  p_i(x)|(-\Delta)^s \varphi(x)|\leqslant\underbrace{|(-\Delta)^s \varphi(x)|\max_{\xi\in\mathbb{R}^n\backslash\Omega}g_y(\xi)}_{:=\psi(x)}\leqslant\left\{
  \begin{aligned}
    &C_3&\quad\text{for}\quad |x|< M,\\
   &\frac{C_4}{|x|^{n+2s}}&\quad\text{for} \quad |x|\geqslant M.
  \end{aligned}\right.
\end{equation}
Therefore, $\psi\in L^1(\mathbb{R}^n)$ and by Lebesgue dominated convergence theorem, we have
\begin{equation}
  \int_{\mathbb{R}^n}S(x)(-\Delta)^s\varphi(x)\intd x=\lim_{i\rightarrow\infty}\int_{\mathbb{R}^n}p_i(x)(-\Delta)^s\varphi(x)\intd x\leqslant0.
\end{equation}
I.e.,
 \begin{equation}
    (-\Delta)^s S\leqslant0\quad\text{in}\quad \mathcal{D}^\prime(\Omega).
  \end{equation}
\end{proof}
\begin{remark}
\label{rmk5}
  As a consequence of Corollary~\ref{cor-ext1} and Theorem~\ref{thm:silvestra-semi}, $S$ has a version $\tilde{S}$ that is upper semi-continuous in $\Omega$. Indeed, for $x\in\Omega$,
  \begin{equation}
    \tilde{S}(x)=\lim_{\lambda\rightarrow 0^+}S\ast\gamma_\lambda(x).
  \end{equation}
\end{remark}
\begin{lemma}[Barrier function]
\label{lem:barrier}
Let $X\in\partial \Omega$, $B_r(c)$ be an exterior ball of $\Omega$ at $X$. Then there exists function $p\in\mathcal{P}$ such that
\begin{equation}
\label{eq:ex-1}
  p(x)=g_y(x)\quad\text{for}\quad x\in \overline{B_r(c)}.
\end{equation}
Moreover, one estimates:
\begin{equation}
\label{est:b}
  |p(x)-p(X)|\leqslant C|x-X|^s.
\end{equation}
Here $C$ depends only on $s,n,r$ and $\operatorname{dist}(y,\partial\Omega)$.
\end{lemma}
\begin{proof}
Doing the Kelvin transform to $g_y$ along $\partial B_r(c)$, one gets
\begin{equation}
\begin{aligned}
\hat{g}_y(x)&=\frac{r^{n-2s} }{|x-c|^{n-2s}} g_y\left(\frac{r^2(x-c)}{|x-c|^2}+c\right) =a(n,s)\frac{r^{n-2s} }{|x-c|^{n-2s}} \left|\frac{r^2(x-c)}{|x-c|^2}+c-y\right|^{2s-n}\\
&=a(n,s)\frac{|\eta-c|^{n-2s}}{r^{n-2s}}\frac{1}{|x-\eta|^{n-2s}},
\end{aligned}
\end{equation}
where
\begin{equation}
  \eta=r^2\frac{(y-c)}{|y-c|^2}+c\in B_r(c).
\end{equation}

One observes $\hat{g}\in\mathcal{L}_{2s}\cap C(\mathbb{R}^n\backslash{B_r(c)})$. Therefore one can apply Lemma~\ref{exi1} to derives the existence of function $\hat{p}\in C(\mathbb{R}^n)$ such that
\begin{equation}
\left\{\begin{array}{lr}
      (-\Delta)^s\hat{p}=0\quad&\text{in}\ B_r(c),\\
      \displaystyle \hat{p}=\hat{g}_y\quad&\text{on}\
      \mathbb{R}^n\backslash B_r(c).
      \end{array}\right.
\end{equation}
In particular, by maximum principle~\ref{thm:silvestre-mp}, we know $\hat{p}\leqslant\hat{g}_y$ in $\mathbb{R}^n$.

Do the Kelvin transform to $\hat{p}$ along $\partial B_r(c)$, one obtains function $p\in C(\mathbb{R}^n)\cap\mathcal{L}_{2s}$ with:
\begin{equation}
\left\{\begin{array}{lr}
      (-\Delta)^sp=0\quad&\text{in}\ \mathbb{R}^n\backslash \overline{B_r(c)},\\
      \displaystyle p=g_y\quad&\text{on}\ \overline{B_r(c)},\\
      \displaystyle p\leqslant g_y\quad&\text{in}\ \mathbb{R}^n.
      \end{array}\right.
\end{equation}
Note that $\Omega \subset\mathbb{R}^n\backslash \overline{B_r(c)}$, we conclude $p\in\mathcal{P}$.
The estimate \eqref{est:b} comes from the fact that
\begin{equation}
  \int_{\mathbb{R}^n\backslash B_r(c)}\frac{\hat{g}_y(x)}{|x-c|^{n+2s}}\intd x\leqslant C_y\qquad\text{and}\qquad [\hat{g}_y]_{C^1(\mathbb{R}^n\backslash B_r(c))}\leqslant C_y.
\end{equation}
\end{proof}

\begin{lemma}
\label{lem:ex-3}
   Let $\Omega$ be a bounded domain in $\mathbb{R}^n$ satisfying the exterior ball condition, $S$ be the Perron solution corresponding to \eqref{pbg}. Then,   $S$ is continuous on $\mathbb{R}^n\backslash\Omega$.
   Moreover,
   \begin{equation}
     S(x)=g_y(x)
     \quad\text{for\ all}\quad x\in\mathbb{R}^n\backslash\Omega.
   \end{equation}
\end{lemma}
\begin{proof}
For $x_1\in \mathbb{R}^n\backslash\overline{\Omega}$, say $\dist(x_1,\Omega)=4\epsilon >0$.
Then we do the following construction:

Let $\Omega^{2\epsilon}=\{x\in\mathbb{R}^n|\dist(x,\Omega)<2\epsilon\}$ and
\begin{equation}
\eta(x)=J_\epsilon[\chi_{\Omega^{2\epsilon}}](x)
\end{equation}
Then $\eta\in C^{\infty}(\mathbb{R}^n)$ and
\begin{equation}
  \eta(x)=\left\{
  \begin{aligned}
0&\quad\text{if}\quad \dist(x,\Omega)\geqslant 3\epsilon\\
1&\quad\text{if}\quad \dist(x,\Omega)\leqslant  \epsilon.
  \end{aligned}
  \right.
\end{equation}
Let $p_1(x)=g_y(x)(1-\eta(x))$. One easily verifies $p_1\in\mathcal{P}$. While, $p_1(x)=g_y(x)$ for $x\in B_\epsilon(x_1)$.
Therefore, $S(x)\geqslant g_y(x)$ for $x\in B_\epsilon(x_1)$. On the other hand, $S(x)\leqslant g_y(x)$ for $x\in B_\epsilon(x_1)$.
Hence, $S(x)=g_y(x)$ for $x\in B_\epsilon(x_1)$ and $S$ is continuous at $x_1$.

For $X\in\partial\Omega$, there exists an exterior ball $B_r(c)$ of $\Omega$ at $X$. According to Lemma \ref{lem:barrier}, there exists $p\in\mathcal{P}$ such that
$p(x)=g_y(x)$ for $x\in \overline{B_r(c)}$.
In particular, $p(X)=g_y(X)$.
On the other hand, $S(x)\leqslant g_y(x)$.
Hence, $S(X)=g_y(X)$. Moreover, in a neighborhood of $X$,
\begin{equation}
  p(x)\leqslant S(x)\leqslant g_y(x),
\end{equation}
with both $p$ and $g_y$ continuous near $X$. As a consequence, $S$ is continuous at $X$.

Summing up the above results, $S$ is continuous on $ \mathbb{R}^n\backslash\Omega$.
Moreover,
   \begin{equation}
     S(x)=g_y(x)
     \quad\text{for\ all}\quad x\in\mathbb{R}^n\backslash\Omega.
   \end{equation}
\end{proof}
\begin{lemma}
  \label{lem:cont}
   Let $\Omega$ be a bounded domain in $\mathbb{R}^n$ satisfying the uniform exterior ball condition, $S$ be the Perron solution corresponding to \eqref{pbg}. Then, $S$ is continuous in $\Omega$.
\end{lemma}
\begin{proof}
$S$ is lower semi-continuous and as a $s$-subharmonic function, $S$ has an upper semi-continuous version in $\Omega$ (see Remark \ref{rmk5}):
\begin{equation}
  \tilde{S}(x)=\lim_{\lambda\rightarrow 0^+}S\ast \gamma_\lambda(x)
\end{equation}
 Hence,  it suffices to show that
  \begin{equation}
  \label{eq:S=S}
  S(x)=\tilde{S}(x)\qquad \text{for\ all}\quad x\in\Omega.
  \end{equation}
  We readily know (see Proposition~\ref{prop:gamma})
    \begin{equation}S(x)=\tilde{S}(x)\qquad\text{for\ a.e.}\ x\in\Omega.
  \end{equation}
  Then, by the semi-continuity of $S$ and $\tilde{S}$,we know:
  \begin{equation}
    S(x)\leqslant\tilde{S}(x).
  \end{equation}
  We now prove that \eqref{eq:S=S} is true by contradiction. Assume that \eqref{eq:S=S} is not true, then there exists $x_0\in\Omega$ such that
  \begin{equation}
  \label{eq:S<S}
    S(x_0)<\tilde{S}(x_0)=\lim_{\lambda\rightarrow 0^+}S\ast \gamma_\lambda(x_0).
  \end{equation}
  We complete the contradiction by the following two claims:
  \begin{enumerate}
    \item[Claim 1:]Assuming \eqref{eq:S<S} is true, then
    \begin{equation}
    \label{eq:JS>S}
      \limsup_{\lambda\rightarrow 0^+}J_\lambda[S](x_0)>S(x_0).
    \end{equation}
    Indeed, if \eqref{eq:JS>S} is not true, then
    \begin{equation}
      \limsup_{\lambda\rightarrow 0^+}J_\lambda[S](x_0)\leqslant S(x_0).
    \end{equation}
    I.e., for any $\epsilon>0$, if $\lambda$ is small enough, then
    \begin{equation}
    \label{eq:upsmi}
      J_\lambda[S](x_0)-S(x_0)\leqslant\epsilon.
    \end{equation}
    On the other hand, $S$ is lower semi-continuous, therefore, for sufficiently small $\lambda>0$,
    \begin{equation}
      S(x)\geqslant S(x_0)-\epsilon\qquad\text{for}\ x\in B_\lambda(x_0).
    \end{equation}
    Now, if we denote $E_\mu=\{x\in B_\mu(x_0)|S(x)>\frac{S(x_0)+\tilde{S}(x_0)}{2}\}$,  then one also calculates that:
    \begin{equation}
    \label{eq:lowsmi}
    \begin{aligned}
      J_\lambda[S](x_0)-S(x_0)&=\int_{B_\lambda(x_0)}j_\lambda(x_0-\eta) (S(\eta)-S(x_0))\intd\eta\\
      &\geqslant \int_{E_{\lambda/2}}j_\lambda(x_0-\eta) \frac{\tilde{S}(x_0)-S(x_0)}{2}\intd\eta -\int_{B_{\lambda}(x_0)\backslash E_{\lambda/2}}\epsilon j_\lambda(x_0-\eta) \intd\eta\\
      &\geqslant\frac{C(\tilde{S}(x_0)-S(x_0))}{2\lambda^n}m(E_{\lambda/2}) -\epsilon.
      \end{aligned}
    \end{equation}
    Combining \eqref{eq:upsmi} and \eqref{eq:lowsmi} and choosing $\epsilon$ sufficiently small, one derives
    \begin{equation}
      \frac{m(E_{\lambda/2})}{m(B_\lambda)}<\frac{1}{2}\qquad\text{for}\ \lambda<\lambda_0.
    \end{equation}
    Denote $F_\lambda=B_\lambda(x_0)\backslash E_\lambda$. Then
    \begin{equation}
     S(x)\leqslant\frac{S(x_0)+\tilde{S}(x_0)}{2} \quad\text{for}\ x\in F_\lambda\qquad\text{and}\qquad \frac{m(F_{\lambda/2})}{m(B_\lambda)}\geqslant\frac{1}{2}\ (\text{when}\ \lambda<\lambda_0).
    \end{equation}
    By the upper semi-continuity of $\tilde{S}$, one derives that for $\mu<\mu_0$ and $x\in B_\mu(x_0)$,
    \begin{equation}\tilde{S}(x)\leqslant \tilde{S}(x_0)+\epsilon.
    \end{equation}
    While, notice that
    \begin{equation}
      \int_{\mathbb{R}^n\backslash B_\mu}\gamma_\lambda(\eta) \intd\eta=\int_{\mathbb{R}^n\backslash B_{\mu/\lambda}}\gamma_1(\eta)\intd\eta.
    \end{equation}
    Hence, $\int_{\mathbb{R}^n\backslash B_{\mu/\lambda}}\gamma_1(\eta)d\eta$ vanishes as $\mu/\lambda\rightarrow \infty$.

    Therefore, by choosing $\mu<\mu_0$ and $\lambda<<\mu$, one estimates:
    \begin{equation}
    \begin{aligned}
      S\ast\gamma_\lambda(x_0)
      &=\int_{\mathbb{R}^n}S(\eta)\gamma_\lambda(x_0-\eta) \intd\eta\\
      &\leqslant \int_{\mathbb{R}^n\backslash B_\mu(x_0)}S(\eta)\gamma_\lambda(x_0-\eta) \intd\eta+
      \int_{ B_\mu(x_0)\backslash F_\lambda}S(\eta)\gamma_\lambda(x_0-\eta) \intd\eta\\
      &+\int_{F_\lambda}S(\eta)\gamma_\lambda(x_0-\eta) \intd\eta\\
      &\leqslant \epsilon+\int_{F_\lambda}\gamma_\lambda(x_0-\eta) \intd\eta\frac{S(x_0)+\tilde{S}(x_0)}{2}+\int_{B_\mu(x_0)\backslash F_\lambda}\gamma_\lambda(x_0-\eta) \intd\eta(\tilde{S}(x_0)+\epsilon)\\
      &\leqslant \tilde{S}(x_0)+2\epsilon-\frac{Cm(F_\lambda)}{\lambda^n} \frac{\tilde{S}(x_0)-S(x_0)}{2}.
      \end{aligned}
    \end{equation}

    Now, by choosing $\epsilon$ small enough, one can derive
    \begin{equation}
      S\ast\gamma_\lambda(x_0)<\tilde{S}(x_0)-\epsilon,
    \end{equation}
and thus contradicts the assumption \eqref{eq:S<S}. This contradiction shows that under the  assumption \eqref{eq:S<S}, one must have \eqref{eq:JS>S}.

    \item[Claim 2]:There exists a function $p\in\mathcal{P}$ such that
    \begin{equation}
      p(x_0)>S(x_0).
    \end{equation}

    Indeed, from Claim 1, we can assume:
    \begin{equation}
     \limsup_{\lambda\rightarrow 0^+}J_\lambda[S](x_0)=\lim_{k\rightarrow\infty}J_{\lambda_k}[S](x_0)= S(x_0)+4h\ (h>0).
    \end{equation}
    As a consequence, there exists $\lambda_1>0$ such that for any $\lambda<\lambda_1$,
    \begin{equation}
      J_\lambda[S](x_0)>S(x_0)+3h.
    \end{equation}

    We now begin the  construction of $p$:

    Seeing Lemma~\ref{lem:barrier}, there exists $\sigma>0$ such that for any $X\in\partial\Omega$, the corresponding barrier function constructed in Lemma~\ref{lem:barrier}, denote as $p_X$, satisfies:
    \begin{equation}
      g_y(x)-p_X(x)\leqslant |g_y(x)-g_y(X)|+|p_X(X)-p_X(x)|<h\qquad\text{for\ any}\quad x\in B_{2\sigma}(X).
    \end{equation}

    For such $\sigma$, there exists $\lambda_2>0$ such that for any $\lambda<\lambda_2$,
    \begin{equation}
    \label{eq:J>g}
      J_\lambda[S](x)\leqslant J_\lambda[g_y](x)\leqslant g_y(x)+h \qquad\text{for\ any}\quad x\in\mathbb{R}^n\backslash \Omega_\sigma.
    \end{equation}

    Since $\overline{\Omega}\backslash\Omega_\sigma$ is compact, and clearly
    \begin{equation}
      \bigcup_{X\in\partial\Omega} B_{2\sigma}(X)\supseteq \overline{\Omega}\backslash\Omega_\sigma,
    \end{equation}
    then there exists a finite subcover, say $\{B_{2\sigma}(X_m)\}_{m=1}^M$ that covers $\overline{\Omega}\backslash\Omega_\sigma$.
    We denote by $p_m=p_{X_m}$ the barrier function corresponding to the boundary points $X_m$.

    We can define a Perron sub-solution:
    \begin{equation}
      p(x)=\max\{J_\lambda[S](x)-2h,0,p_1(x),\cdots,p_M(x)\}.
    \end{equation}
    One easily observe that for $\lambda<\lambda_1$,
    \begin{equation}
      p(x_0)\geqslant J_\lambda[S](x_0)-2h> S(x_0)+3h-2h=S(x_0)+h>S(x_0).
    \end{equation}
    We now verify that for $\lambda<\lambda_0=:\min\{\lambda_1,\lambda_2,\sigma\}$, $p$ indeed is a Perron subsolution:
    \begin{itemize}
    \item $p$  is continuous, since it is the maximum of finitely many continuous functions.
      \item Clearly, $p\geqslant 0$.
      \item In fact, from \eqref{eq:J>g}, it clearly holds that:
    \begin{equation}
      J_\lambda[S](x)-2h\leqslant g_y(x)
      \qquad\text{for}\quad x\in\mathbb{R}^n\backslash \Omega.
    \end{equation}
    As a consequence,
    \begin{equation}
      p(x)\leqslant g_y(x)\qquad\text{for}\quad x\in\mathbb{R}^n\backslash \Omega.
    \end{equation}
    \item Observe the fact that for any $x\in\Omega\backslash \Omega_\sigma$,
there is $1\leqslant m\leqslant M$, such that $x\in B_{2\sigma}(X_m)$, and hence there is $p_m$, such that
\begin{equation}
  J_\lambda[S](x)-2h\leqslant g_y(x)-h<p_m(x).
\end{equation}
In sight of Theorem~\ref{thm:maxcons}, one derives the following:
\begin{equation}
  (-\Delta)^s p\leqslant 0\qquad\text{in}\quad \mathcal{D}^\prime(\Omega).
\end{equation}
 \end{itemize}
Claim 2 is proved now.
  \end{enumerate}
  Claim 2 clearly contradicts the definition of $S$. This contradiction indicates that \eqref{eq:S=S} must be true.
\end{proof}
Now we are ready to prove Theorem~\ref{thm:ex-1}, indeed, we are going to show $h_y=S$ is exactly the solution of problem \eqref{pbg}.
\begin{proof}[Proof of Theorem~\ref{thm:ex-1}]
Seeing Lemma~\ref{lem:cont}, we only to prove the following two claims:
\begin{itemize}
  \item $S$ is $s$-harmonic in $\Omega$
  \item $S$ is smooth in $\Omega$
\end{itemize}

  Seeing the fact that,
  \begin{equation}
    (-\Delta)^sS\leqslant 0\qquad\text{in}\quad\mathcal{D}^\prime(\Omega),
  \end{equation}
  there exists a non-negative Radon measure $\mu$ on $\Omega$ such that:
    \begin{equation}
    (-\Delta)^sS=-\mu\qquad\text{in}\quad\mathcal{D}^\prime(\Omega).
  \end{equation}
  We prove $\mu=0$ by contradiction.
  Suppose $\mu\neq 0$, then there exist a ball $B_{3r}(x_0)\subset \Omega$ such that $\mu(B_r(x_0))>0$.

  For $\lambda<r$, we can define
  \begin{equation}
    v_\lambda(x)=J_\lambda[S](x)+\int_{B_r(x_0)}J_\lambda[\Phi](x-y)\intd \mu_y.
  \end{equation}
  One calculates:
  \begin{equation}
    v_\lambda(x_0)\geqslant J_\lambda[S](x_0)+C_*r^{2s-n}\mu(B_r(x_0)).
  \end{equation}
  Clearly,
  \begin{equation}
    (-\Delta)^sv_\lambda\leqslant 0\qquad\text{in}\quad\mathcal{D}^\prime(\Omega_\lambda).
  \end{equation}
  While
  \begin{equation}
    v_\lambda(x)\leqslant J_\lambda[S](x)+C_*(2r)^{2s-n}\mu(B_r(x_0))\qquad\text{for} \quad x\in \mathbb{R}^n\backslash\Omega.
  \end{equation}
  Hence,
    \begin{equation}
    v_\lambda(x)-C_*(2r)^{2s-n}\mu(B_r(x_0))\leqslant J_\lambda[S](x)\qquad\text{for} \quad x\in \mathbb{R}^n\backslash\Omega.
  \end{equation}
  and
    \begin{equation}
    v_\lambda(x)-C_*(2r)^{2s-n}\mu(B_r(x_0))\geqslant J_\lambda[S](x_0)+C_sC_*r^{2s-n}\mu(B_r(x_0)).
  \end{equation}
  Then, one can use the same technique as in Claim 2 of Lemma~\ref{lem:cont} to construct a Perron sub-solution $p$ such that $p(x_0)>S(x_0)$. A contradiction!

  This contradiction shows that $S$ must be $s$ harmonic in $\Omega$.

The smoothness of $S$ in $\Omega$ follows immediately according to Lemma 4.1.
\end{proof}

\section{Properties of Green functions}
\label{sec.4}
In this section, we always assume $\Omega$ to be a bounded domain in $\mathbb{R}^n$ satisfying the uniform exterior ball condition, with $r$ being the uniform radius of exterior balls.

With Green function $G_\Omega(\cdot,\cdot)$ readily defined, it immediately follows from the Definition~\ref{def:GO} that
  \begin{corollary}
  \label{cor:delta}
    For any fixed $y\in\Omega$, $G_\Omega(\cdot, y)\in C(\mathbb{R}^n\backslash\{y\})\cap C^\infty(\Omega\backslash\{y\})$ and satisfies
    \begin{equation}
      (-\Delta)^s G_\Omega(\cdot,y)=\delta_y\qquad\text{in}\quad\mathcal{D}^\prime(\Omega).
    \end{equation}
  \end{corollary}
  \begin{proposition}
  \label{prop:estLp}
    For any fixed $y\in\Omega$ and any $1\leqslant p<\frac{n}{n-2s}$, $G_\Omega(\cdot,y)\in L^p(\Omega)$ with
    \begin{equation}
      \|G_\Omega(\cdot,y)\|_{L^p(\Omega)}\leqslant C(n,s,p,\diam(\Omega)).
    \end{equation}
  \end{proposition}
  Proposition~\ref{prop:estLp} follows immediately from the following inequality:
  \begin{equation}
    0\leqslant G_\Omega(x,y)\leqslant \Phi(x-y).
  \end{equation}
  \begin{lemma}[Barrier function]
  \label{lem:barrierG}
    Let $X\in\partial\Omega$, $B_r(c)$ be an exterior ball of $\Omega$ at $X$. Then for any $x,y\in\Omega$, it holds that
    \begin{equation}
    \label{ineq:barrier}
      G_\Omega(x,y)\leqslant G_{\mathbb{R}^n\backslash\overline{B_r(c)}}(x,y).
      \end{equation}
      Here
      \begin{equation}
       G_{\mathbb{R}^n\backslash\overline{B_r(c)}}(x,y)= \frac{\kappa(n,s)}{|x-y|^{n-2s}} \int_{0}^{\rho_E(x,y)}\frac{t^{s-1}}{(1+t)^{\frac{n}{2}}}\intd t,
    \end{equation}
    for
    \begin{equation}
      \rho_E(x,y)=\frac{(|x-c|^2 -r^2)(|y-c|^2-r^2)}{r^2|x-y|^2}.
    \end{equation}
  \end{lemma}
  \begin{proof}
    A Kelvin transformation criteria shows that to $G_{\mathbb{R}^n\backslash\overline{B_r(c)}}(x,y)$ is the Green function of  $\mathbb{R}^n\backslash\overline{B_r(c)}$ in the sense that
    \begin{equation}
      (-\Delta)^s G_{\mathbb{R}^n\backslash\overline{B_r(c)}}(\cdot,y)=\delta_y \qquad\text{in}\quad \mathcal{D}^\prime(\mathbb{R}^n\backslash\overline{B_r(c)}).
    \end{equation}
    Note the fact that both $G_{\mathbb{R}^n\backslash\Omega}(\cdot,y)$ and $G_{\mathbb{R}^n\backslash\overline{B_r(c)}}(\cdot,y)$ are continuous away from $y$, and satisfies:
    \begin{equation}
    \begin{aligned}
      (-\Delta)^sG_\Omega(\cdot,y)=\delta_y=(-\Delta)^s G_{\mathbb{R}^n\backslash\overline{B_r(c)}}(\cdot,y) &\qquad\text{in}&\quad \mathcal{D}^\prime(\mathbb{R}^n\backslash\Omega),\\
         G_\Omega(x,y)=0\leqslant G_{\mathbb{R}^n\backslash\overline{B_r(c)}}(x,y)&\qquad\text{for}&\quad x\in\mathbb{R}^n\backslash \Omega.
    \end{aligned}
    \end{equation}
    Then \eqref{ineq:barrier} is guaranteed by the maximum principle~\ref{thm:silvestre-mp}.
  \end{proof}
  \begin{proposition}
  \label{prop:estG}
  For any $x,y\in\Omega$, the following inequality holds:
    \begin{equation}
      G_\Omega(x,y)\leqslant C(n,s,r,\diam(\Omega))\frac{\dist(x,\partial \Omega)^s}{|x-y|^{n-s}}.
    \end{equation}
  \end{proposition}
\begin{proof}
  Fix $x\in\Omega$, there exists $X\in\partial\Omega$ such that
  \begin{equation}
    |x-X|=\dist(x,\partial\Omega).
  \end{equation}
  Let $B_r(c)$ be an exterior ball of $\Omega$ at $X$.
  Then, we estimate as follow:
  \begin{enumerate}
    \item If $|x-y|\leqslant |x-X|$, then
    \begin{equation}
    \begin{aligned}
      |y-c|^2-r^2&\leqslant C(r,\diam(\Omega))(|y-c|-r)\\
      &\leqslant C(r,\diam(\Omega))(|x-X|+|x-y|) \leqslant C(r,\diam(\Omega))|x-y|.
    \end{aligned}
    \end{equation}
    Then
    \begin{equation}\rho_E(x,y)\leqslant C(r,\diam(\Omega))\frac{|x-X|}{|x-y|}.\end{equation}
    As a consequence of Lemma~\ref{lem:barrierG},
  \begin{equation}
  \begin{aligned}
    G_\Omega(x,y)&\leqslant G_{\mathbb{R}^n\backslash\overline{B_r(c)}}(x,y)\\
    &\leqslant C(n,s)\frac{1}{|x-y|^{n-2s}}(\rho_E(x,y))^s\\
    &\leqslant C(n,s,r,\diam(\Omega))\frac{|x-X|^s}{|x-y|^{n-s}}.
    \end{aligned}
  \end{equation}
  \item If $|x-y|>|x-X|$, then
  \begin{equation}
  \begin{aligned}
    G_\Omega(x,y)&\leqslant G_{\mathbb{R}^n\backslash\overline{B_r(c)}}(x,y)\\
    &\leqslant \kappa(n,s)\frac{1}{|x-y|^{n-2s}}\int_{0}^\infty \frac{t^{s-1}}{(1+t)^\frac{n}{2}}\intd t\\
    &=: C(n,s)\frac{1}{|x-y|^{n-2s}}
\leqslant C(n,s)\frac{|x-X|^s}{|x-y|^{n-s}}.
    \end{aligned}
  \end{equation}
  \end{enumerate}
Noting $|x-X|=\dist(x,\partial\Omega)$, we finish the proof.
\end{proof}
  \begin{remark}
  \label{rmk7}
    Similarly, one can also prove  for any $x,y\in\Omega$:
    \begin{equation}
      G_\Omega(x,y)\leqslant C(n,s,r,\diam(\Omega))\frac{\dist(y,\partial \Omega)^s}{|x-y|^{n-s}}.
    \end{equation}

\begin{definition}
  Analogously to the definition \eqref{Gf}, we denote
  \begin{equation}
    G_\Omega\ast f(x):= \int_{\mathbb{R}^n} f(y)G_\Omega(x,y)\intd y= \int_{\Omega} f(y)G_\Omega(x,y)\intd y.
  \end{equation}
\end{definition}

  \end{remark}
\begin{theorem}
\label{thm:Gf}
  For any $f\in C(\overline{\Omega})$, problem
\begin{equation}\label{pbf}
\left\{
\begin{aligned}
(-\Delta)^s u&=f\qquad&\text{in}\quad\Omega,\\
u&=0\qquad&\text{on}\quad\partial \Omega.
\end{aligned}
\right.
\end{equation}
  has a unique continuous solution: $u=G_\Omega\ast f$.
\end{theorem}
  \begin{proof}
    For arbitrary $\varphi\in \mathcal{D}(\Omega)$, one calculates:
    \begin{equation}
      \begin{aligned}
        (-\Delta)^su[\varphi]=&\int_{\mathbb{R}^n}G_\Omega\ast f(x)\cdot (-\Delta)^s\varphi(x)\intd x\\
        =&\int_{\Omega}G_\Omega\ast f(x)\cdot (-\Delta)^s\varphi(x)\intd x\\
        =&\int_{\Omega} \int_{\Omega}G_\Omega(x,y)f(y)(-\Delta)^s\varphi(x)\intd y \intd x\\
        =&\int_{\Omega}\left[ \int_{\Omega}G_\Omega(x,y)(-\Delta)^s\varphi(x)\intd x\right] f(y)\intd y\\
        =&\int_{\Omega} \varphi(y)f(y)\intd y.\\
      \end{aligned}
    \end{equation}
    I.e.,
    \begin{equation}
      (-\Delta)^s u=f\qquad\text{in}\quad\mathcal{D}^\prime(\Omega).
    \end{equation}
    It is easy to see that $u$ is continuous on $\mathbb{R}^n\backslash \overline{\Omega}$ and $\Omega$. Hence, we only need to show the continuity of $u=G_\Omega\ast f$ on $\partial\Omega$.

    Indeed, for $x\in\Omega$,
    \begin{equation}
    \begin{aligned}
        |u(x)|=&\left|\int_{\Omega}G_\Omega(x,y) f(y)\intd y\right|\\
        \leqslant &\int_{\Omega}G_\Omega(x,y)|f(y)|\intd y\\
        \leqslant&C\dist(x,\partial\Omega)^s\int_{\Omega} \frac{|f(y)|}{|x-y|^{n-s}}\intd y\\
        \leqslant& C\dist(x,\partial\Omega)^s\rightarrow 0\qquad\text{as}\quad x\rightarrow\partial\Omega.\\
      \end{aligned}
    \end{equation}
    The continuity is proved. The maximum principle~\ref{thm:silvestre-mp} ensures the uniqueness of solution.
  \end{proof}
  \begin{lemma}
  \label{lem:invdelta}
  For fixed $x\in\Omega$,
    \begin{equation}
    \label{eq:invdelta}
      (-\Delta)^s G_\Omega(x,\cdot)=\delta_x\qquad\text{in}\quad \mathcal{D}^\prime(\Omega).
    \end{equation}
  \end{lemma}
  \begin{proof}
    For any $\varphi\in \mathcal{D}(\Omega)$, we have $(-\Delta)^s\varphi\in C(\overline{\Omega})$.
    Let $u=G_\Omega\ast[(-\Delta)^s\varphi]$.
    Then  Theorem~\ref{thm:Gf} shows that $u$ is the only continuous function satisfying
\begin{equation}
\left\{
\begin{aligned}
(-\Delta)^s u&=(-\Delta)^s\varphi \qquad&\text{in}\quad\Omega,\\
u&=0=\varphi\qquad&\text{on}\quad\partial \Omega.
\end{aligned}
\right.
\end{equation}
Then $u(x)=\varphi(x)$, which indicates \eqref{eq:invdelta}.
  \end{proof}
\begin{theorem}
  $G_\Omega(\cdot,\cdot)$ is symmetric, i.e.
  \begin{equation}
  \label{eq:symmG}
    G_\Omega(x,y)=G_\Omega(y,x).
  \end{equation}
\end{theorem}
\begin{proof}
For fixed $y\in\Omega$, let $v(x)=G_\Omega(x,y)-G_\Omega(y,x)$.
Then Corollary~\ref{cor:delta} and Lemma~\ref{lem:invdelta} shows that
\begin{equation}
  (-\Delta)^sv=0\qquad\text{in}\quad\mathcal{D}^\prime(\Omega).
\end{equation}
As a consequence, $v$ is smooth in $\Omega$. Note that $v(x)=0$ for $x\in\mathbb{R}^n\backslash\Omega$.
And Remark~\ref{rmk7} also ensures the continuity of $v$ on $\partial\Omega$. Then we know $v$ is continuous in $\mathbb{R}^n$.

The maximum principle~\ref{thm:silvestre-mp} then shows that, $v(x)=0$, i.e. \eqref{eq:symmG} holds.
\end{proof}

\section{Poisson kernel}
\label{sec.5}
In this section, we still assume $\Omega$ to be a bounded domain in $\mathbb{R}^n$ satisfying the uniform exterior ball condition, with $r$ being the uniform radius of exterior balls.
We also need the following terminologies of domain shifting:
Let $\Omega$ be a bounded domain in $\mathbb{R}^n$, we denote
\begin{equation}
  \Omega_\epsilon=\{x\in\Omega\mid\dist(x,\partial\Omega)>\epsilon\} =\{x\in\Omega\mid \overline{B_\epsilon(x)}\subset\Omega\},
\end{equation}
\begin{equation}
  \Omega^\epsilon=\{x\in\mathbb{R}^n\mid\dist(x,\Omega)<\epsilon\} =\bigcup_{x\in\Omega}B_\epsilon(x) =\bigcup_{x\in\Omega}\overline{B_\epsilon(x)}.
\end{equation}

\begin{definition}
Define the Poisson kernel of $\Omega$ to be
\begin{equation}
  P_\Omega(x,y)=C_{n,s}\int_\Omega \frac{G_\Omega(x,z)}{|z-y|^{n+2s}}\intd z\qquad x\in\Omega\quad y\in\mathbb{R}^n\backslash\Omega.
\end{equation}
\end{definition}
\begin{definition}
  Analogously to the definition \eqref{Pg}, we denote
  \begin{equation}
    P_\Omega\ast g(x):= \begin{cases}
      \displaystyle\int_{\mathbb{R}^n\backslash\Omega} g(y)P_\Omega(x,y)\intd y &\qquad\text{for}\quad x\in\Omega,\\
      g(x)&\qquad\text{for}\quad x\in\mathbb{R}^n\backslash \Omega.
    \end{cases}
  \end{equation}
\end{definition}
Then one instantly observes that:
\begin{theorem}
\label{thm:Pg0}
  For any $g\in C^\infty(\mathbb{R}^n)\cap\mathcal{L}_{2s}$ with $g=0$ in $\Omega$, problem
  \begin{equation}\label{pb:og}
\left\{
\begin{aligned}
(-\Delta)^s u&=0\qquad&\text{in}\quad&\Omega,\\
u&=g\qquad&\text{on}\quad&\partial \Omega.
\end{aligned}
\right.
\end{equation}
has a unique continuous solution: $u=P_\Omega\ast g$.
\end{theorem}
\begin{proof}
  Indeed, one can calculate for $x\in\Omega$:
  \begin{equation}
  \begin{aligned}
    (-\Delta)^sg(x)&=C_{n,s}\int_{\mathbb{R}^n}\frac{g(x)-g(y)}{|x-y|^{n+2s}}\intd y\\
    &=-C_{n,s}\int_{\mathbb{R}^n\backslash\Omega}\frac{g(y)}{|x-y|^{n+2s}}\intd y.
  \end{aligned}
  \end{equation}
  Clearly $(-\Delta)^s g\in C(\overline{\Omega})$. Let $v=G_\Omega\ast[(-\Delta)^sg]$ then Theorem~\ref{thm:Gf} shows that $v\in C(\mathbb{R}^n)$ with
  \begin{equation}
  \begin{cases}
    (-\Delta)^s v=(-\Delta)^s g&\qquad\text{in}\quad\Omega,\\
    v=0&\qquad\text{on}\quad\mathbb{R}^n\backslash\Omega.
    \end{cases}
  \end{equation}
  Hence letting $u=g-v$, we know $u$ is a continuous solution of \eqref{pb:og}.
  For $x\in\Omega$, one calculates
  \begin{equation}
    \begin{aligned}
      u(x)&=g(x)-v(x)\\
      &=0-\int_{\Omega}G(x,z)(-\Delta)^s g(z)\intd z\\
      &=\int_{\Omega}G(x,z)\left[C_{n,s} \int_{\mathbb{R}^n\backslash\Omega}\frac{g(y)}{|z-y|^{n+2s}}\intd y\right]\intd z\\
      &=C_{n,s}\int_{\mathbb{R}^n\backslash\Omega}g(y)\left[\int_\Omega \frac{G_\Omega(x,z)}{|z-y|^{n+2s}}\intd z \right]\intd y\\
      &=\int_{\mathbb{R}^n\backslash\Omega}g(y)P_\Omega(x,y)\intd y.
    \end{aligned}
  \end{equation}
  The uniqueness comes immediately from the maximum principle~\ref{thm:silvestre-mp}.
\end{proof}
With Theorem~\ref{thm:Pg0} proved, we can estimate:
\begin{lemma}[Barrier function]
\label{lem:barrierP}
  Let $B_r(c)$ be an exterior ball of $\Omega$, for $x\in\Omega$, $y\in B_r(c)$, one estimates:
  \begin{equation}
    P_\Omega(x,y)\leqslant c(n,s)\left(\frac{|x-c|^2-r^2}{r^2-|y-c|^2}\right)^s\frac{1}{|x-y|^n}.
  \end{equation}
\end{lemma}
\begin{proof}
A Kelvin transform argument instantly shows that
\begin{equation}
  P_{\mathbb{R}^n\backslash \overline{B_r(c)}}(x,y)=c(n,s)\left(\frac{|x-c|^2-r^2}{r^2-|y-c|^2}\right)^s\frac{1}{|x-y|^n}
\end{equation}
performs as the Poisson kernel of $\mathbb{R}^n\backslash \overline{B_r(c)}$.

Then, for any non-negative function $\varphi\in \mathcal{D}(\mathbb{R}^n)$ with $\operatorname{supp}(\varphi)\subset\subset B_r(c)$, one has:
\begin{equation}
\begin{cases}
  (-\Delta)^s P_\Omega\ast\varphi=0=(-\Delta)^s P_{\mathbb{R}^n\backslash \overline{B_r(c)}}\ast\varphi
  &\qquad\text{in}\quad \mathcal{D}^\prime(\Omega),\\
  P_\Omega\ast\varphi=\varphi=P_{\mathbb{R}^n\backslash \overline{B_r(c)}}\ast\varphi
  &\qquad\text{in}\quad B_r(c),\\
  P_\Omega\ast\varphi=0\leqslant P_{\mathbb{R}^n\backslash \overline{B_r(c)}}\ast\varphi
  &\qquad\text{in}\quad \mathbb{R}^n\backslash (B_r(c)\cup \Omega).
  \end{cases}
\end{equation}
Moreover, seeing Theorem~\ref{thm:Pg0}, we know both $P_\Omega\ast\varphi$ and $P_{\mathbb{R}^n\backslash \overline{B_r(c)}}\ast\varphi$ are continuous.
Then the maximum principle \ref{thm:silvestre-mp} gives that, for any given $x\in\Omega$,
\begin{equation}
   \int_{B_r(c)}P_\Omega(x,y)\varphi(y)\intd y\leqslant \int_{B_r(c)} P_{\mathbb{R}^n\backslash \overline{B_r(c)}}(x,y)\varphi(y)\intd y.
\end{equation}
This indicates that
\begin{equation}
  P_\Omega(x,y) \leqslant P_{\mathbb{R}^n\backslash \overline{B_r(c)}}(x,y)\qquad\text{for\ any}\quad x\in\Omega\quad\text{and}\quad y\in B_r(c).
\end{equation}
\end{proof}
\begin{lemma}
\label{lem:estP1}
 For $x\in\Omega$, $y\in\Omega^r\backslash\overline{\Omega}$, it holds that
   \begin{equation}
    P_\Omega(x,y)\leqslant \frac{C(n,s,r, \diam\Omega) }{\dist(y,\partial\Omega)^s|x-y|^{n-s}}.
  \end{equation}
\end{lemma}
\begin{proof}
  Let $X\in\partial\Omega$ such that $|y-X|=\dist(y,\partial\Omega)$, let $B_r(c)$ be the exterior ball of $\Omega$ at $X$. Then clearly we have $y\in B_r(c)$, and hence Lemma \ref{lem:barrierP} shows that
  \begin{equation}
    P_\Omega(x,y)\leqslant c(n,s)\left(\frac{|x-c|^2-r^2}{r^2-|y-c|^2}\right)^s\frac{1}{|x-y|^n}.
  \end{equation}
  Here
  \begin{equation}
    r^2-|y-c|^2\geqslant r(r-|y-c|)=r|y-X|=r\dist(y,\partial\Omega).
  \end{equation}
   \begin{equation}
    |x-c|^2-r^2\leqslant (2r+\diam(\Omega))(|x-c|-r)\leqslant (2r+\diam(\Omega))|x-y|.
  \end{equation}
  Then  we have
  \begin{equation}
    P_\Omega(x,y)\leqslant C\frac{1}{\dist(y,\partial\Omega)^s}\frac{1}{|x-y|^{n-s}}.
  \end{equation}
\end{proof}
\begin{lemma}
\label{lem:estP2}
  For $x\in\Omega$, $y\in\mathbb{R}^n\backslash\Omega$, it holds that
  \begin{equation}
    P_\Omega(x,y)\leqslant C(n,s,\Omega) \frac{\dist(x,\partial\Omega)^s}{\dist(y,\partial\Omega)^{2s}|x-y|^{n-s}}.
  \end{equation}
\end{lemma}
\begin{proof}
Let $r=\frac{|x-y|}{2}$
  From Proposition~\ref{prop:estG}, we has
  \begin{equation}
    G_\Omega(x,z)\leqslant C(n,s,\Omega)\frac{\dist(x,\partial\Omega)^s}{|x-z|^{n-s}}.
  \end{equation}
  Therefore
  \begin{equation}
  \begin{aligned}
    P_\Omega(x,y) &=\int_{\Omega}G_\Omega(x,z)\frac{1}{|z-y|^{n+2s}}\intd z\\
    &\leqslant C\int_{\Omega} \frac{\dist(x,\partial\Omega)^s}{|x-z|^{n-s}|z-y|^{n+2s}} \intd z\\
    &=C\Biggl[
    \int_{\Omega\backslash B_r(x)} \frac{\dist(x,\partial\Omega)^s}{|x-z|^{n-s}|z-y|^{n+2s}} \intd z
    +
    \int_{B_r(x)} \frac{\dist(x,\partial\Omega)^s}{|x-z|^{n-s}|z-y|^{n+2s}} \intd z
    \Biggr]\\
    &\leqslant C\Biggl[
    \int_{\Omega\backslash B_r(x)} \frac{\dist(x,\partial\Omega)^s}{r^{n-s}|z-y|^{n+2s}} \intd z+
    \int_{B_r(x)} \frac{\dist(x,\partial\Omega)^s}{|x-z|^{n-s}(2r-|x-z|)^{n+2s}} \intd z
    \Biggr]\\
    &\leqslant C\Biggl[ \frac{\dist(x,\partial\Omega)^s}{r^{n-s}\dist(y,\partial\Omega)^{2s}} +\frac{\dist(x,\partial\Omega)^s}{r^{n+s}}\Biggr]\leqslant C\frac{\dist(x,\partial\Omega)^s}{r^{n-s}\dist(y,\partial\Omega)^{2s}}.
    \end{aligned}
  \end{equation}
\end{proof}
\begin{lemma}
\label{lem:estP3}
  For  $x\in\Omega$, $y\in\mathbb{R}^n\backslash\Omega^r$, one estimates:
  \begin{equation}
    P_\Omega(x,y)\leqslant C(n,s,\Omega)\frac{\dist(x,\partial\Omega)^s}{|x-y|^{n+2s}}.
  \end{equation}
\end{lemma}
\begin{proof}
  For $x\in\Omega$ and $y\in\mathbb{R}^n\backslash \Omega$ with $\dist(y,\partial\Omega)>r$, one can choose $X\in\partial\Omega$ such that $\dist(y,\partial\Omega)=|y-X|$, then
  \begin{equation}
   |y-x|\leqslant |y-X|+|X-x|\leqslant \left(1+\frac{\diam(\Omega)}{r}\right)|y-X| =\left(1+\frac{\diam(\Omega)}{r}\right)\dist(y,\partial\Omega).
  \end{equation}
  and also $|y-z|\geqslant\dist(y,\partial\Omega)$ for any $z\in\Omega$.
  Therefore
  \begin{equation}
  \begin{aligned}
    P_\Omega(x,y)&=C_{n,s}\int_\Omega \frac{G_\Omega(x,z)}{|z-y|^{n+2s}}\intd z\\
    &\leqslant C(n,s,\Omega)\int_\Omega \frac{1}{\dist(y,\partial\Omega)^{n+2s}} \frac{\dist(x,\partial\Omega)^s}{|x-z|^{n-2s}}\intd z\\
    &\leqslant \frac{\dist(x,\partial\Omega)^s}{|x-y|^{n+2s}} \int_{B_{\diam(\Omega)}(x)}\frac{C(n,s,\Omega)}{|x-z|^{n-2s}}\intd z\\
    &=C(n,s,\Omega)\frac{\dist(x,\partial\Omega)^s}{|x-y|^{n+2s}}.
    \end{aligned}
  \end{equation}

\end{proof}

\begin{lemma}
\label{lem:doubleball}
  Let $R\geqslant1$, $c$ be a point in $\mathbb{R}^n$ with $|c|=1+R$, $v\in L^\infty(\mathbb{R}^n)$ satisfies:
  \begin{equation}
  \begin{cases}
    (-\Delta)^s v=0&\qquad\text{in}\quad B_1=B_1(0),\\
    v=0&\qquad\text{on}\quad B_R(c).
    \end{cases}
  \end{equation}
  Then for a universal positive constant $\eta_0<1$,
  \begin{equation}
   |v(0)|\leqslant \eta_0\|v\|_{L^\infty(\mathbb{R}^n)}.
  \end{equation}
\end{lemma}
\begin{proof}
We only need to prove for $R=1$ and $|c|=2$.
  Let $w=P\ast v$, as defined in \eqref{Pg}, then
  \begin{equation}
  \begin{cases}
    (-\Delta)^s w=0=(-\Delta)^s v&\qquad\text{in}\quad B_1,\\
    w=v&\qquad\text{on}\quad \mathbb{R}^n\backslash B_1.
    \end{cases}
  \end{equation}
  Moreover, both $w$ and $v$ are bounded. Hence, Theorem indicates that
  \begin{equation}
    w\equiv v.
  \end{equation}
  \begin{equation}
  \begin{aligned}
    |v(0)|&=|P\ast v(0)|\\
    &\leqslant \int_{\mathbb{R}^n\backslash B_1} P(0,y)|v(y)|\intd y\\
    &\leqslant \int_{\mathbb{R}^n\backslash (B_1\cup B_1(c))} P(0,y)\intd y \|v\|_{L^\infty(\mathbb{R}^n)}\\
    &=\underbrace{\left(1-\int_{B_1(c)} P(0,y)\intd y\right)}_{\eta_0}\|v\|_{L^\infty(\mathbb{R}^n)}.
    \end{aligned}
  \end{equation}
\end{proof}
\begin{proposition}
  For any $x\in\Omega$, the following identity holds:
  \begin{equation}
    \int_{\mathbb{R}^n\backslash\Omega}P_\Omega(x,y)\intd y=1.
  \end{equation}
\end{proposition}
\begin{proof}
Let
\begin{equation}
  u(x)=P_\Omega\ast\chi_{\mathbb{R}^n\backslash\overline{\Omega}}=
  \begin{cases}
    \displaystyle\int_{\mathbb{R}^n\backslash\Omega}P_\Omega(x,y)\intd y&\qquad\text{for}\quad x\in\Omega,\\
    \displaystyle1&\qquad\text{for}\quad x\in\mathbb{R}^n\backslash\Omega.
  \end{cases}
\end{equation}
It suffices to show that $u\equiv 1$.

Indeed, one constructs
  \begin{equation}
    g_k(x)=\int_{\dist(y,\overline{\Omega})\geqslant1/k}j_{1/k}(x-y)\intd y,
  \end{equation}
  where $j_\epsilon$ is the mollifier.
  Then each $g_k\in C^\infty(\mathbb{R}^n)\cap\mathcal{L}_{2s}$, $g_k=0$ in $\Omega$ and
  \begin{equation}
    g_k\nearrow \chi_{\mathbb{R}^n\backslash\overline{\Omega}}\qquad\text{as}\quad k\rightarrow \infty.
  \end{equation}
  The monotone convergence theorem immediately shows that
  \begin{equation}
    u=\lim_{k\rightarrow\infty}P_\Omega\ast g_k.
  \end{equation}
  Theorem~\ref{thm:Pg0} implies that
  \begin{equation}
    (-\Delta)^s[P_\Omega\ast g_k]=0\qquad\text{in}\quad\mathcal{D}^\prime(\Omega).
      \end{equation}
  then the maximum principle~\ref{thm:silvestre-mp} also gives:
  \begin{equation}
    0\leqslant P_\Omega\ast g_k\leqslant 1.
  \end{equation}
  As the point-wise limit of $P_\Omega\ast g_k$, $u$ satisfies:
    \begin{equation}
    (-\Delta)^su=0\qquad\text{in}\quad\mathcal{D}^\prime(\Omega), \qquad\text{and}\qquad    0\leqslant u\leqslant 1.
  \end{equation}
  Consider $v=1-u$, we know:
  \begin{equation}
   0\leqslant v\leqslant 1,  \qquad (-\Delta)^sv=0\quad\text{in}\quad\mathcal{D}^\prime(\Omega) \qquad\text{and}\qquad v=0\quad\text{on}\quad\mathbb{R}^n\backslash\Omega.
  \end{equation}
  As a consequence, $v\in C^\infty(\Omega)$.
  We now prove $v=0$.

  For any $x\in\Omega$ with $\dist(x,\partial\Omega)\leqslant r$, we may say
  $X\in\overline{B_{\dist(x,\partial\Omega)}(x)}\cap\partial\Omega$. Let $B_r(c)$ be the exterior ball at $X$, then we have:
  \begin{equation}
    (-\Delta)^sv=0\quad\text{in}\quad \mathcal{D}^\prime(B_{\dist(x,\partial\Omega)}(x)) \qquad\text{and}\qquad v=0\quad\text{on}\quad \overline{B_r(c)}.
  \end{equation}
  Doing a translation and scaling, one concludes from Lemma~\ref{lem:doubleball} that
  \begin{equation}
    v(x)\leqslant \eta_0 \|v\|_{L^\infty(\mathbb{R}^n)}=\eta_0 \|v\|_{L^\infty(\Omega)}.
  \end{equation}
  Hence, $v\leqslant \eta_0 \|v\|_{L^\infty(\Omega)}$ on $\mathbb{R}^n\backslash \Omega_r$.
  Since $v$ is continuous in $\Omega\supset \overline{\Omega_r}$, one applies maximum principle~\ref{thm:silvestre-mp} to $v$ on $\Omega_r$ to derive
  \begin{equation}
    v(x)\leqslant \eta_0\|v\|_{L^\infty(\Omega)}\qquad\text{for}\quad x\in\Omega_r.
  \end{equation}
  Therefore
  \begin{equation}
    \|v\|_{L^\infty(\Omega)}\leqslant \eta_0\|v\|_{L^\infty(\Omega)}.
  \end{equation}
  Seeing that $\eta_0<1$, we know $v=0$. Hence, $u\equiv 1$.
\end{proof}

\begin{lemma}
\label{lem:contPg}
  For $g\in C(\mathbb{R}^n\backslash\Omega)\cap \mathcal{L}_{2s}$, $P_\Omega\ast g$ is continuous in $\mathbb{R}^n$.
\end{lemma}
\begin{proof}
$P_\Omega\ast g=g$ is continuous in $\mathbb{R}^n\backslash\overline{\Omega}$.
 The continuity of $P_\Omega\ast g$ in $\Omega$ is also clear.
  It suffices to show that $u$ is continuous on $\partial\Omega$.

  Indeed, for any $X\in\partial\Omega$ and $x\in\Omega$,
  \begin{equation}
  \begin{aligned}
    |P_\Omega\ast g(x)-g(X)| &= \left|\int_{\mathbb{R}^n\backslash\Omega}P_\Omega(x,y)g(y)\intd y-g(X)\right|\\
    &= \left|\int_{\mathbb{R}^n\backslash\Omega} P_\Omega(x,y)(g(y)-g(X))\intd y\right|\\
    &\leqslant\int_{\mathbb{R}^n\backslash\Omega} P_\Omega(x,y)|g(y)-g(X)|\intd y
    \end{aligned}
  \end{equation}
  For any $\epsilon>0$, the continuity of $g$ assures that there exists $\sigma>0$ such that
  \begin{equation}
    |g(y)-g(X)|<\frac{\epsilon}{3}\qquad\text{for\ any}\quad y\in B_{2\sigma}(X)\backslash\Omega.
  \end{equation}
Write
\begin{equation}
\begin{aligned}
 &\int_{\mathbb{R}^n\backslash\Omega} P_\Omega(x,y)|g(y)-g(X)|\intd y\\
  =&\left[\int_{B_{2\sigma}(X)\backslash\Omega}
  +\int_{\Omega^r\backslash(\Omega\cup B_{2\sigma}(X))}+\int_{\mathbb{R}^n\backslash\Omega^r}\right]
  P_\Omega(x,y)|g(y)-g(X)|\intd y\\
  =&I_1+I_2+I_3.
  \end{aligned}
\end{equation}
For $I_1$, one instantly estimates:
  \begin{equation}
I_1\leqslant \frac{\epsilon}{3}\int_{B_{2\sigma}(X)\backslash\Omega}P_\Omega(x,y) \intd y<\frac{\epsilon}{3}\int_{\mathbb{R}^n\backslash\Omega}P_\Omega(x,y) \intd y=\frac{\epsilon}{3}.
  \end{equation}
Utilizing Lemma~\ref{lem:estP3}, one estimates:
    \begin{equation}
    \label{eq:Pg=g1}
I_3\leqslant C\int_{\mathbb{R}^n\backslash \Omega^r}\frac{\dist(x,\partial\Omega)^s(1+|g(y)|)}{|x-y|^{n+2s}}\intd y
<C\dist(x,\partial\Omega)^s.
  \end{equation}
  While for $I_2$, we restrict $x\in B_\sigma(X)\cap \Omega$ and estimate in the following way:
  \begin{equation}
    I_2\leqslant C\int_{\Omega^r\backslash( \Omega\cup B_{2\sigma}(X))}P_\Omega(x,y)\intd y.
  \end{equation}
According to Lemma~\ref{lem:estP1}, for $x\in\Omega\cap B_\sigma(X)$ and $y\in\Omega^r\backslash(\overline{\Omega}\cup B_{2\sigma}(X))$,
\begin{equation}
     P_\Omega(x,y)\leqslant \frac{C}{\dist(y,\partial\Omega)^s|x-y|^{n-s}}\leqslant \frac{C}{\sigma^{n-s}}{\dist(y,\partial\Omega)}^{-s}.
\end{equation}
One observes that $\dist(y,\partial\Omega)^{-s}$ is integrable in $\Omega^r\backslash(\overline{\Omega}\cup B_{2\sigma}(X))$.

On the other hand, According to Lemma~\ref{lem:estP2}, $P_\Omega(x,y)\rightarrow 0$ as $\dist(x,\partial\Omega)\rightarrow 0^+$.
Hence, the dominated convergence theorem ensures that
\begin{equation}
\label{eq:Pg=g2}
  I_2\rightarrow 0\qquad\text{as}\quad \dist(x,\partial\Omega)\rightarrow 0^+.
\end{equation}

Combining \eqref{eq:Pg=g1} and \eqref{eq:Pg=g2}, one derives that there is $\delta>0$ such that for any $x\in B_\sigma(X)\cap\Omega$ with $\dist(x,\partial\Omega)<\delta$, we have $I_2<\frac{\epsilon}{3}$ and $I_3<\frac{\epsilon}{3}$, and hence
\begin{equation}
  \int_{\mathbb{R}^n\backslash\Omega} P_\Omega(x,y)|g(y)-g(X)|\intd y
  =I_1+I_2+I_3<\epsilon.
\end{equation}
This shows that $P_\Omega\ast g(x)\rightarrow g(X)$ if $x\rightarrow X$.
\end{proof}
\begin{theorem}
\label{thm:Pg}
  For $g\in C(\mathbb{R}^n\backslash\Omega)\cap \mathcal{L}_{2s}$, problem~\eqref{pb:og} has a unique continuous solution: $u=P_\Omega\ast g$.
\end{theorem}
\begin{proof}
  An approximation method instantly shows that $u=P_\Omega\ast g$ satisfies:
  \begin{equation}
    (-\Delta)^s u=0\qquad\text{in}\quad\mathcal{D}^\prime(\Omega).
  \end{equation}
  Lemma~\ref{lem:contPg} shows that $u=P_\Omega\ast g$ is continuous.
  The uniqueness comes immediately from the maximum principle~\ref{thm:silvestre-mp}.
\end{proof}

\section{Dirichlet problem}
\label{sec.6}
 In this section, we first establish the existence and uniqueness theorem for the problem
\begin{equation}
  \label{pb:pure}
\left\{
\begin{aligned}
(-\Delta)^s u&=f\quad &\text{in}\ &\Omega,\\
u&=0\quad&\text{on}\ &\mathbb{R}^n\backslash\Omega.
\end{aligned}
\right.
\end{equation}
After that, we extend these results to problem \eqref{pb+c}.

One observes the following lemma, which describes the persistence of exterior ball condition when shifting the domain.
\begin{lemma}
  Let $\Omega$ be a bounded domain in $\mathbb{R}^n$ satisfying the uniform exterior ball condition, with $r$ being the uniform radius of exterior balls. Then one conclude the following:
  \begin{enumerate}
    \item $\Omega_\epsilon$ also satisfies the uniform exterior ball condition with uniform radius $r+\epsilon$.
    \item If $\epsilon<r$ then $\Omega^\epsilon$ satisfies the uniform exterior ball condition with uniform radius $r-\epsilon$.
  \end{enumerate}
\end{lemma}
Hence, we are ready to show the following uniqueness theorem:
\begin{theorem}[Uniqueness]
\label{thm:un}
  Let $\Omega$ be a bounded domain in $\mathbb{R}^n$ satisfying the uniform exterior ball condition, $0<s<1$, $f=0$ and suppose  $u\in \mathcal{L}_{2s}$ is a solution of \eqref{pb:pure} satisfying \eqref{cond0}, then $u=0$ in $\Omega$.
\end{theorem}
\begin{proof}
First we do the mollify, let $u_\epsilon(x)=J_\epsilon u(x)$. Then we know $u_\epsilon\in C_0^\infty(\mathbb{R}^n)$ satisfies the following equations
\begin{equation}
\left\{\begin{aligned}
(-\Delta)^s u_\epsilon&=0\ \ &\text{in}\ &\Omega_{\epsilon}\\
u_\epsilon&=0\ \ &\text{in}\ &\mathbb{R}^n\backslash \Omega^{\epsilon}
\end{aligned}\right.
\end{equation}

Then we can  represent $u$,  by using Theorem~\ref{thm:Pg} on $\Omega_\epsilon$, as
\begin{equation}
u_\epsilon(x)
=\int_{\mathbb{R}^n\backslash \Omega_{\epsilon}}P_{\Omega_\epsilon}(x,y)u_\epsilon(y)\ dy,\ \text{for}\ x\in B_{1-\epsilon}.
\end{equation}
Here $P_{\Omega_\epsilon}(x,y)$ denotes the Poisson kernel of $\Omega_\epsilon$.

Now for any given $U\subset\subset\Omega$, let $3\sigma=\dist(U,\partial\Omega)$. We choose $0<\epsilon<\sigma$ and estimate $\|u_\epsilon\|_{L^\infty(U)}$. Indeed for $x\in U$,
\begin{equation}
\label{ieq0}
\begin{aligned}
|u_\epsilon(x)|
&=\left|\int_{\mathbb{R}^n\backslash \Omega_\epsilon} P_{\Omega_\epsilon}(x,y)u_\epsilon(y)\intd y\right|\\
&=\left|\int_{\Omega^\epsilon\backslash \Omega_\epsilon} P_{\Omega_\epsilon}(x,y)\left[\int_{\Omega\backslash \Omega_{2\epsilon}}u(z)j_\epsilon(y-z)\intd z\right] \intd y\right|\\
&\leqslant \int_{\Omega\backslash \Omega_{2\epsilon}}\underbrace{\left[\int_{\Omega^\epsilon\backslash \Omega_\epsilon} P_{\Omega_\epsilon}(x,y)j_\epsilon(y-z)\intd y\right]}_{\text{denote\ as\ }A_\epsilon(x,z)}|u(z)|\intd z.
\end{aligned}
\end{equation}

We claim that for a positive constant $M$ depending only on $n,s,\Omega$ and  $U$, the following inequality holds:
\begin{equation}
\label{ieq0-1}
  A_\epsilon(x,z)\leqslant M\epsilon^{-s}\qquad\text{for\ any}\quad x\in U\quad\text{and}\quad z\in\Omega\backslash \Omega_{2\epsilon}.
\end{equation}
In fact, from Lemma~\ref{lem:estP1}, we know
   \begin{equation}
    P_{\Omega_\epsilon}(x,y)\leqslant \frac{C }{\dist(y,\partial\Omega_\epsilon)^s|x-y|^{n-s}}\leqslant \frac{C}{\sigma^{n-s}\dist(y,\partial\Omega_\epsilon)^s}\qquad\text{for}\quad x\in U\quad y\in\Omega^\epsilon\backslash\Omega_\epsilon.
  \end{equation}
  Then
  \begin{equation}
\begin{aligned}
A_\epsilon(x,z)
&\leqslant  M_0\epsilon^{-n}\int_{\Omega^\epsilon\backslash \Omega_\epsilon} \dist(y,\partial\Omega_\epsilon)^{-s}\chi_{B_\epsilon(z)}\intd y\\
&=M_0\epsilon^{-n}\left(\int_{\Omega\backslash \Omega_\epsilon} +\int_{\Omega^\epsilon\backslash \Omega}\right) \dist(y,\partial\Omega_\epsilon)^{-s}\chi_{B_\epsilon(z)}\intd y\\
&\leqslant M_0\epsilon^{-n}\int_{0}^{\epsilon}\int_{\partial\Omega_{\epsilon-t}\cap B_\epsilon(z)}\intd \mathcal{H}^{n-1}_yt^{-s}\intd t+M_0\epsilon^{-n}\epsilon^{-s}|B_\epsilon(z)|\\
&\leqslant M\epsilon^{-s}.
\end{aligned}
\end{equation}
We have proved the claim \eqref{ieq0-1}.
Now, substituting \eqref{ieq0-1} into \eqref{ieq0}, one has
\begin{equation}\label{ieq1}
\|u_\epsilon\|_{L^\infty(U)}\leqslant M\epsilon^{-s}\int_{\Omega\backslash \Omega_{2\epsilon}} |u(z)|\ dz.
\end{equation}

Noting the condition \eqref{cond0}, by letting $\epsilon\rightarrow0^+$ in \eqref{ieq1}, one obtain:
\begin{equation}\label{eq1}
\|u\|_{L^\infty(U)}=0.
\end{equation}

Since $U$ in \eqref{eq1} is arbitrary, one can derive $\|u\|_{L^\infty(\Omega)}=0$.
\end{proof}
\begin{remark}
\label{rmk:mp}
  A similar approach also shows that under same conditions, if $f\leqslant 0$ then $u\leqslant 0$.
\end{remark}
Based on this, one can also prove Theorem~\ref{thm:main-mp} as follow:
\begin{proof}[Proof of Theorem~\ref{thm:main-mp}]
  Consider $u^-=\max\{-u,0\}$, Theorem~\ref{thm:maxcons} shows that
  \begin{equation}
 \left\{ \begin{aligned}
    (-\Delta)^s u^-\leqslant -cu^-\leqslant 0\qquad&\text{in}\quad\mathcal{D}^\prime(\Omega),\\
    u^-=0\qquad&\text{on}\quad\mathbb{R}^n\backslash\Omega.
    \end{aligned} \right.
  \end{equation}
  Then Remark~\ref{rmk:mp} instantly gives $u^-\leqslant0$, hence $u^-=0$ and $u\geqslant0$ in $\Omega$.
\end{proof}
\begin{theorem}[Existence]\label{thm:exis}
Let $\Omega$ be a bounded domain in $\mathbb{R}^n$ satisfying the uniform exterior ball condition, $0<s<1$ and suppose $f\in  L^1_s(B_1)$. Then $u=G_\Omega\ast f$ solves the problem \eqref{pb1}. Furthermore, $u$ satisfies the uniqueness condition \eqref{cond0}.
\end{theorem}
\begin{proof}
We first show that such $u$ satisfies the condition \eqref{cond0}.

Indeed, for $\epsilon>0$, one calculates:
\begin{equation}
  \frac{1}{\epsilon^s}\int_{\Omega\backslash \Omega_{\epsilon}}|u(x)| \intd x
  \leqslant \int_{\Omega}\underbrace{\left[\frac{1}{\epsilon^s} \int_{\Omega\backslash \Omega_{\epsilon}} G_\Omega(x,y)\intd x\right]}_{\text{denote\ as}\ A_\epsilon(y)}|f(y)|\intd y.
\end{equation}
Utilizing Proposition~\ref{prop:estG}, one estimates $A_\epsilon$ as follow:
\begin{itemize}
  \item If $y\in \Omega\backslash \Omega_{2\epsilon}$
  \begin{equation}
    \begin{aligned}
    A_\epsilon (y)&\leqslant \frac{1}{\epsilon^s} \int_{\Omega\backslash \Omega_{4\epsilon}} G_\Omega(x,y)\intd x\\
    &\leqslant\frac{1}{\epsilon^s}\left[
    \int_{\Omega\backslash (\Omega_{4\epsilon}\cup B_{\dist(y,\partial\Omega)}(y))} G_\Omega(x,y)\intd x
    +\int_{B_{\dist(y,\partial\Omega)}(y)} G_\Omega(x,y)\intd x
    \right]\\
    &\leqslant\frac{C}{\epsilon^s}\left[
    \int_{\Omega\backslash (\Omega_{4\epsilon}\cup B_{\dist(y,\partial\Omega)}(y))} \frac{\dist(y,\partial\Omega)^s}{|x-y|^{n-s}}\intd x
    +\int_{B_{\dist(y,\partial\Omega)}(y)} \frac{1}{|x-y|^{n-2s}}\intd x
    \right]\\
    &\leqslant \frac{C}{\epsilon^s}\int_{0}^{4\epsilon}\int_{\partial \Omega_t\backslash B_{\dist(y,\partial\Omega)}(y) }\frac{t^s}{|x-y|^{n-s}}\intd\mathcal{H}^{n-1}_x\intd t+C\dist(y,\partial\Omega)^s\\
    &\leqslant \frac{C}{\epsilon^s}\left[\int_{0}^{2\dist(y,\partial\Omega)} \frac{t^s}{\dist(y,\partial\Omega)^{1-s}}\intd t +\int_{2\dist(y,\partial\Omega)}^{4\epsilon} \frac{t^s}{(t-\dist(y,\partial\Omega))^{1-s}}\intd t \right]\\
    &\qquad\qquad+C\dist(y,\partial\Omega)^s \\
    &\leqslant C\dist(y,\partial\Omega)^s.
    \end{aligned}
  \end{equation}
  \item If  $y\in \Omega_{2\epsilon}$
  \begin{equation}
  \begin{aligned}
    A_\epsilon (y)&\leqslant \frac{C}{\epsilon^s}\int_{\Omega\backslash \Omega_{\epsilon}}\frac{\dist(x,\partial\Omega)^s}{|x-y|^{n-s}}\intd x\\
    &=\frac{C}{\epsilon^s} \int_{0}^\epsilon\int_{\partial \Omega_{t}}\frac{t^s}{|x-y|^{n-s}}\intd\mathcal{H}^{n-1}_xdt\\
    &\leqslant \frac{C}{\epsilon^s}\int_{0}^\epsilon\frac{t^s}{\epsilon^{1-s}}dt\quad\left(\text{since}\ \operatorname{dist}(y,\partial \Omega_t)\geqslant\epsilon\right)\\
    &\leqslant C\epsilon^s\leqslant C\dist(y,\partial\Omega)^s.
    \end{aligned}
  \end{equation}
\end{itemize}
Therefore, $A_\epsilon(y)\leqslant C\dist(y,\partial\Omega)^s$ and $A_\epsilon(y)\rightarrow 0$ as $\epsilon\rightarrow 0$.

Now, we know $A_\epsilon(y)|f(y)|\leqslant C\dist(y,\partial\Omega)^s|f(y)|\in L^1(\Omega)$ and $A_\epsilon(y)|f(y)|\rightarrow 0$ as $\epsilon\rightarrow 0$, the Lebesgue convergence theorem assures that
\begin{equation}
  \limsup_{\epsilon\rightarrow 0}\frac{1}{\epsilon^s}\int_{\Omega\backslash \Omega_{\epsilon}}|u(x)| dx\leqslant \lim_{\epsilon\rightarrow 0}\int_{\Omega}A_\epsilon(y)|f(y)|dy=0.
\end{equation}
I.e. $u$ satisfies the uniqueness condition \eqref{cond0}.

The above calculation also gives us a basic estimate:
\begin{equation}
\label{est4}
  \|G\ast f\|_{L^1(\Omega)}\leqslant C\|f\|_{L^1_s(\Omega)}.
\end{equation}

Since $f\in  L^1_s(\Omega)$, we can choose $f_k\in C^\infty(\overline{\Omega})\cap L^1_s(B_1)$ that converges to $f$ in $ L^1_s(\Omega)$ as $k\rightarrow\infty$.
According to Theorem~\ref{thm:Gf}, there exists a corresponding solution sequence $\{u_k=G_\Omega\ast f_k\}$.
\eqref{est4} implies that $u_k$ converges to $u$ in $L^1(\Omega)$ as $k\rightarrow\infty$. Hence, $u$ as the distributional limit of $u_k$ satisfies \eqref{pb:pure}.
\end{proof}
\begin{remark}
Theorem~\ref{thm:main1} is a direct corollary of Theorem \ref{thm:un} and Theorem \ref{thm:exis}.
\end{remark}

\section{Acknowledgements}
Both authors were partially supported by National Natural Science Foundation of China (Grant Nos.
12031012, 11831003) and Natural Science Foundation of Henan Province of China (Grant No.222300420499).

\end{document}